\documentclass[a4paper,12pt]{amsart}

\usepackage{graphicx}
\usepackage{amsmath,amssymb,amsthm,amsfonts}
\usepackage{amssymb}

\newtheorem{thm}{Theorem}[section]
\newtheorem{cor}[thm]{Corollary}

\newtheorem{prop}[thm]{Proposition}
\newtheorem{rem}[thm]{Remark}
\newtheorem{exe}[thm]{Example}
\newtheorem{defn}[thm]{Definition}
\newcommand{\bexe}{\begin{exe} \textup}
\newcommand{\eexe}{\end{exe} }
\newcommand{\bremark}{\begin{rem} \textup}
\newcommand{\eremark}{\end{rem} }

\newcommand{\cuad}{{\sqcap\kern-.68em\sqcup}}

\newcommand{\R}{{\mathbb{R}}}
\newcommand{\N}{{\mathbb{N}}}

\begin{document}

\title [Nonexistence of stable solutions ]{Nonexistence of stable solutions to quasilinear elliptic equations on Riemannian manifolds   }
       \date{}
       \author[]{D.D. Monticelli, F. Punzo and B. Sciunzi}

\thanks{
{\em D.D. Monticelli} -- Politecnico di Milano -- (Dipartimento di Matematica)
-- V. Bonardi 9, Milano, Italy. Email: {\tt
dario.monticelli@polimi.it}, partially supported by GNAMPA Project 2016 ``Strutture speciali e PDEs in geometria Riemanniana''.\\
{\em F. Punzo} -- Universit\`a della Calabria -- (Dipartimento di
Matematica e Informatica) -- V. P. Bucci, 31 B -- Rende
(CS), Italy. Email: {\tt fabio.punzo@unical.it}, partially supportedy by GNAMPA Project 2016 ``Propriet\`a qualitative di soluzioni di equazioni ellittiche e paraboliche non lineari''. \\
{\em B. Sciunzi} -- Universit\`a della Calabria -- (Dipartimento di Matematica e Informatica) -- V. P. Bucci, 31 B -- Rende (CS), Italy. Email: {\tt sciunzi@mat.unical.it}. Partially supported by GNAMPA Project 2016 ``Propriet\`a qualitative di soluzioni di equazioni ellittiche e paraboliche non lineari'', and by {\em MIUR Me\-to\-di va\-ria\-zio\-na\-li ed equa\-zio\-ni dif\-fe\-ren\-zia\-li non\-li\-nea\-ri}.
}


\begin{abstract}
 We prove nonexistence of nontrivial, possibly sign changing, stable solutions to a class of quasilinear elliptic equations with a potential on Riemannian manifolds, under suitable weighted growth conditions on geodesic balls.
\end{abstract}

 \maketitle

\section{Introduction}
In this paper we investigate nonexistence of nontrivial, possibly sign changing, {\em stable} solutions of equation
\begin{equation}\label{eq-1}
-\Delta_p\, u=V(x)|u|^{\sigma-1}u \quad\quad \text{in } M,
\end{equation}
where $M$ is a  complete, non-compact Riemannian manifold of
dimension $m$ endowed with a metric tensor $g$, $V$ is a given positive function, $\Delta_p$ denotes the $p-$Laplace operator, that is
\[\Delta_p u:=\operatorname{div}\left\{ |\nabla u|^{p-2} \nabla u \right\}, \]
where $\operatorname{div}$ and $\nabla$ are the divergence operator and the gradient associated to the metric $g$, respectively. In particular, for $p=2, \Delta_p$ becomes the standard Laplace-Beltrami operator on $M$.
Concerning the parameters $\sigma$ and $p$, and the potential $V$, we always assume that
\[ \sigma> p-1\geq 1,\]
and that
\[ V\in L^1_{\textrm{loc}}(M), \quad V> 0 \quad \textrm{a.e. in}\;\; M\,. \]
The precise notions of solutions and of stability are given in Definitions \ref{def1} and \ref{defst} below, respectively.

\smallskip
Starting from the seminal paper \cite{FA2},
for $M=\mathbb R^m$, $V\equiv 1$ solutions of equation \eqref{eq-1} have been largely studied in the literature (see, e.g., \cite{Bi}-\cite{GS2}, \cite{JL}, \cite{MitPohoz359}-\cite{MitPohozMilan}). In particular, in  \cite{FA2} it is shown that if $M=\mathbb R^m$, $V\equiv 1$, $p=2$ and
\begin{equation}\label{e1ia}
\begin{cases}
1<\sigma<+\infty & \textrm{if}\;\; m\leq 10\,, \\
1<\sigma<\sigma_c(m):=\frac{(m-2)^2 - 4m + 8\sqrt{m-1}}{(m-2)(m-10)} & \textrm{if}\;\; m\geq 11\,,
\end{cases}
\end{equation}
then the unique stable solution of equation \eqref{eq-1} is $u\equiv 0$. Furthermore, the previous result has been generalized to the case  $p>2$ in \cite{DFSV}. Indeed, when $M=\mathbb R^m$, $V\equiv 1$, $p>2$, it is established that if $u$ is a stable solution of equation \eqref{eq-1} and
  \begin{equation}\label{e1i}
\begin{cases}
p-1<\sigma<+\infty & \textrm{if}\;\; m\leq \frac{p(p+3)}{p-1}\,, \\
p-1<\sigma<\sigma_c(m, p) & \textrm{if}\;\; m> \frac{p(p+3)}{p-1}\,,
\end{cases}
\end{equation}
with
\begin{equation}\label{e5i}
 \sigma_c(m, p):=\frac{[(p-1)m -p]^2 + p^2(p-2) - p^2(p-1)m + 2p^2\sqrt{(p-1)(m-1)}}{(m-p)[(p-1)m - p(p+3)]}\,,
\end{equation}
then $u\equiv 0$ in $\mathbb R^m$. We should observe that in \cite{FA2} and in \cite{DFSV} more regular solutions are used; specifically, solutions of class $C^2$ and $C^{1, \alpha}_{\rm loc}$, respectively, are considered. However, in view of standard regularity results (see \cite{Di}), weak solutions meant in the sense of Definition \ref{def1} belong to those classes, e.g. if $V$ is locally bounded.

\smallskip

On the other hand, nonexistence of {\em positive} solutions (not necessarily stable) of equation \eqref{eq-1}  has also been studied on Riemannian manifold (see, e.g., \cite{GrigKond}, \cite{GrigS}, \cite{MaMoPu1}, \cite{Sun1}, \cite{Sun2}).  More precisely, denote by $d\mu$  the canonical Riemannian volume element on $M$; fix any reference point $o\in M$ and set $B_R:=\{x\in M\,:\, \operatorname{dist}(x, o) < R \}\,.$ Furthermore, let
\begin{equation}\label{e3i}
\alpha_0:= \frac{p\sigma}{\sigma-p+1}\,,\quad \beta_0:=\frac{p-1}{\sigma -p +1}\,.
\end{equation}
In \cite{MaMoPu1} it is proved that equation \eqref{eq-1} does not admit any nontrivial positive solution, provided that there exist $C_0>0$, $k\in (0, \beta_0)$
such that, for every $R>0$ sufficiently large and for every $\varepsilon>0$ sufficiently small,
\begin{equation}\label{e4i}
\int_{B_R\setminus B_{\frac R 2}} V^{-\beta_0 +\varepsilon} \, d\mu \leq C R^{\alpha_0 + C_0\varepsilon} (\log R)^k\,.
\end{equation}
The same conclusion remains true, if instead of \eqref{e4i} it is assumed that there exist $C_0>0$
such that, for every $R>0$ sufficiently large and for every $\varepsilon>0$ sufficiently small,
\[ \int_{B_R\setminus B_{\frac R 2}} V^{-\beta_0 +\varepsilon} \, d\mu \leq C R^{\alpha_0 + C_0\varepsilon} (\log R)^{\beta_0}\,,\]
and
\[  \int_{B_R\setminus B_{\frac R 2}} V^{-\beta_0 -\varepsilon} \, d\mu \leq C R^{\alpha_0 + C_0\varepsilon} (\log R)^{\beta_0}\,. \]
Such results are in agreement with those in  \cite{MitPohozMilan} for $M=\mathbb R^m$. In fact, in this case no nonnegative, nontrivial solution exists, provided that
\[ V\equiv 1, \quad m>p, \quad p-1<\sigma<\frac{m(p-1)}{m-p}\,.\]
Similar results have  been also obtained for parabolic equations (see, e.g., \cite{MaMoPu2}, \cite{MitPohozMilan}, \cite{PoTe})\,.
\medskip

The aim of this paper is to study nonexistence of {\em stable}, possibly sign changing solutions of equation \eqref{eq-1} on Riemannian manifolds. We show that no nontrivial stable solution exists, provided that a condition similar to \eqref{e4i} is satisfied (see Theorems \ref{thm1} and \ref{thm2} below). Note that for $M=\mathbb R^m$, $V\equiv 1$, our results are in accordance with those stated in \cite{DFSV} (see Example \ref{ex1} below). In order to prove our results, first we deduce suitable a priori estimates for stable solutions of equation \eqref{eq-1}, by using the definition of stable weak solutions (see Propositions \ref{PR:1} and \ref{PR:2} below). Then by choosing an appropriate family of test functions, depending on two parameters, in such a priori estimates, we infer that any stable solution $u$ of equation \eqref{eq-1} vanishes identically.

The estimate in Proposition \ref{PR:1} is similar to that in \cite[Proposition 1.4]{DFSV} (where $V\equiv 1$, $M=\mathbb R^m$); however, it is more accurate. In particular, we explicitly compute the presence of $\delta=\bar{\gamma}-\gamma$ in the r.h.s. of the integral estimate \eqref{EQ:10bis}, see also \eqref{defgamma} and \eqref{defdelta}; this will be expedient in order to use the appropriate test functions in the sequel. On the other hand, the estimate in Proposition \ref{PR:2} is new.

Let us mention that we do not use the same test functions as in \cite{DFSV}. Indeed, by doing so, weaker results than ours will be obtained. We use the same family of test functions, depending on two parameters, employed in \cite{GrigS}, \cite{MaMoPu1}. However, in \cite{GrigS}, \cite{MaMoPu1} such test functions  were used in different a priori estimates, therefore in our situation in order to get the conclusion new estimates are necessary.

\medskip

The paper is organized as follows. In Section \ref{MRes} we state our main results, which are then proved in Section \ref{PrMRes}. Section \ref{ARes} contains some important auxiliary tools, that are used in the proofs of the main theorems.

\section{Statement of the main results}\label{MRes}
Denote by $\operatorname{Lip}_{\rm loc}(M)$ the set of Lipschitz functions
$\varphi:M\rightarrow\R$ having compact support.
For any open domain $\Omega\subset M$ and $p>1$, let
$W^{1,p}(\Omega)$ be the completion with respect to the norm
$$\|w\|_{W^{1,p}(\Omega)}=\left(\int_\Omega|w|^p\,d\mu+\int_\Omega|\nabla w|^p\,d\mu\right)^\frac{1}{p}$$ of the space of locally Lipschitz
functions $w:\Omega\rightarrow\R$ having finite $W^{1,p}(\Omega)$
norm. For any function $u:M\rightarrow\R$, we say that $u\in
W^{1,p}_\text{loc}(M)$ if for every $\varphi \in \operatorname{Lip}_{\rm loc}(M)$ one has that
$u\varphi\in W^{1,p}(M)$.

\begin{defn}\label{def1}
We say that $u\in W^{1,p}_{\rm loc}(M)\cap L^\sigma_{\rm loc}(M)$
is a weak solution of \eqref{eq-1} if, for every
$\varphi\in \operatorname{Lip}_{\rm loc}(M) $, one has
\begin{equation}\label{EQ:1}
\displaystyle\int_M |\nabla u|^{p-2}(\nabla u,\nabla \varphi)\,
d\mu=\int_{M}V(x) |u|^{\sigma-1}u \varphi\,d\mu\,,
\end{equation}
where $(\cdot,\cdot)$ is the scalar product
associated to the metric $g$.
\end{defn}
We explicitly note that the solution $u$ in the above definition can
change sign and can be unbounded on $M$. Moreover, it is well known
that the class of test functions used in Definition \ref{def1} can
be equivalently restricted to $\varphi\in C^\infty_c(M)$.

\smallskip

The {\it linearized operator} of \eqref{EQ:1} at $u$ is given by
\begin{equation*}
\begin{split}
L_u (v,\varphi)\,=& \int_M|\nabla u|^{p-2}(\nabla v,\nabla
\varphi)\,d\mu+\\
&+ \int_M \big\{ (p-2)|\nabla u|^{p-4}(\nabla u,\nabla v)\,g(\nabla
u,\nabla \varphi) -\sigma V(x)|u|^{\sigma-1} v
\varphi\,\big\} d\mu\end{split}
\end{equation*}
for every $v,\varphi \in \operatorname{Lip}_{loc}(M)$.

\begin{defn}\label{defst}
We say that a weak solution $u\in W^{1,p}_{\rm loc}(M)\cap
L^\sigma_{\rm loc}(M)$ of equation \eqref{eq-1} is {\em stable}, if
$$
L_u(\varphi,\varphi) \geq 0\qquad \forall \varphi \in \operatorname{Lip}_{loc}(M)\,.
$$
\end{defn}

\medskip

To state the following result we need to introduce some notation. We
set
\begin{equation}\label{defgamma}
\bar\gamma\,=\,\bar\gamma_{p,\sigma}\,:=\,
\frac{2\sigma-p+1+2\sqrt{\sigma(\sigma-p+1)}}{p-1}
\end{equation}
and, for any $\gamma\in\,[ \frac{\sigma}{p-1}\,,\,\bar\gamma)$ we set
\begin{equation}\label{defdelta}
\delta\,=\,\delta_\gamma\,:=\,\bar\gamma-\gamma.
\end{equation}

Let $r(x)$ denote the geodesic distance of a point $x\in M$ from a fixed origin $o\in M$ and for every $R>0$ let $B_R$ be the geodesic ball with radius $o$ and radius $R$.

\begin{defn}
Let
\begin{equation}\label{alfabeto}
\alpha:=p\frac{\sigma+\overline\gamma}{\sigma-p+1}\qquad\beta:=\frac{\overline\gamma+p-1}{\sigma-p+1}\,.
\end{equation}
We say that condition {\bf (HP1)} holds if there exist positive
constants $C$, $C_0, R_0, \epsilon_0$ such that for every $\epsilon\in (0, \epsilon_0)$
and $R\in (R_0, \infty)$ one has
\begin{equation}\label{volHP2}
\int_{B_R\setminus B_{R/2}}V^{-\beta+\epsilon}\,d\mu \leq
CR^{\alpha+C_0\epsilon}(\log R)^{-1}.
\end{equation}
We say that condition {\bf (HP2)} holds if there exist positive
constants $C$, $C_0, R_0, \epsilon_0$, $\theta$ and a constant $b<-1+\frac{\theta}{\sigma-p+1}$ such that for every $\epsilon\in (0, \epsilon_0)$
and $R\in (R_0, \infty)$ one has
\begin{equation}\label{31}
\int_{B_R\setminus B_{R/2}}V^{-\beta+\epsilon}\,d\mu \leq
CR^{\alpha+C_0\epsilon}(\log R)^{b}e^{-\theta\epsilon \log\frac{R}{2}\log\log \frac{R}{2}}.
\end{equation}
\end{defn}

Let us observe that
$$\alpha >\alpha_0\,, \quad \beta>\beta_0,$$
where $\alpha_0, \beta_0$ are defined in \eqref{e3i}\,.

\begin{rem}\label{rem1}
We note that condition {\bf (HP1)} holds if there exist positive
constants $C$, $C_0$ such that
\begin{equation}\label{1}
0<V(x)\leq C(1+r(x))^{C_0}\qquad\text{ in }M
\end{equation}
and
\begin{equation}\label{2}
\int_{B_R\setminus B_{R/2}}V^{-\beta}\,d\mu \leq CR^{\alpha}(\log
R)^{-1}
\end{equation}
for every $R>0$ sufficiently large.

Similarly, condition {\bf (HP2)} holds if for some positive constants $C$, $C_0$, $\theta$ one
has
\begin{equation}\label{29}
0<V(x)\leq Cr(x)^{C_0}e^{-\theta \log r(x) \log\log r(x)}
\end{equation}
for every $x\in M\setminus B_{R^*}$ for some fixed $R^*>0$  and
\begin{equation}\label{30}
\int_{B_R\setminus B_{R/2}}V^{-\beta}\,d\mu \leq CR^{\alpha}(\log
R)^{b}
\end{equation}
for some $b<-1+\frac{\theta}{\sigma-p+1}$ and every $R>0$ sufficiently large.
\end{rem}

Now we can state our main nonexistence results.

\begin{thm}\label{thm1}
Suppose that condition {\bf (HP1)} holds. Let $u$ be a stable solution of
equation \eqref{eq-1}, then $u\equiv 0$ in $M$.
\end{thm}

\begin{thm}\label{thm2}
Suppose that condition {\bf (HP2)} holds. Let $u$ be a stable solution of
equation \eqref{eq-1}, then $u\equiv 0$ in $M$.
\end{thm}

By Theorem \ref{thm2} we immediately have the following corollary.
\begin{cor}\label{cor1}
  If for some positive constants $C$, $C_0$, $\theta$ one
has
\begin{equation}\label{41}
0<V(x)\leq C(1+r(x))^{C_0}e^{-\theta r(x)}\qquad\text{ in }M
\end{equation}
and \eqref{30} for some $b\in\R$ and every $R>0$ large enough.  Let $u$ be a stable solution of equation \eqref{eq-1}. Then
$u\equiv 0$ in $M$.
\end{cor}

\begin{exe}\label{ex1}{\em
In the case $M=\R^m$ and $V\equiv1$, condition {\bf (HP1)} is satisfied if and only if
  \begin{equation}\label{3}
  m<\alpha=\alpha(\sigma)=p\frac{\sigma+\overline\gamma(\sigma)}{\sigma-p+1},
   \end{equation}
 with $\overline\gamma(\sigma)$ defined in \eqref{defgamma}. It is easy to see that $\alpha(\sigma)$ is a strictly decreasing function for $\sigma\in(p-1,\infty)$, with $$\lim_{\sigma\rightarrow (p-1)^+}\alpha(\sigma)=+\infty,\qquad\lim_{\sigma\rightarrow+\infty}\alpha(\sigma)=\frac{p(p+3)}{p-1}.$$
Therefore, if $m\leq\frac{p(p+3)}{p-1}$, condition \eqref{3} is satisfied for every $\sigma>p-1$. On the other hand, if
 $m>\frac{p(p+3)}{p-1}$, there exists a unique $\sigma^*\in(p-1,\infty)$ such that $\alpha(\sigma^*)=m$ and condition \eqref{3} is satisfied if and only if $\sigma\in(p-1,\sigma^*)$. An easy computation shows that
 \[  \sigma^* = \sigma_c(m, p),\]
 with $\sigma_c(m, p)$ defined in \eqref{e5i}.
Hence if $V\equiv1$ and \eqref{e1i} hold, from Theorem \ref{thm1} it follows that
 every weak solution $u$ of equation \eqref{eq-1} which is stable in $\R^n$
vanishes identically. Note that the result is in accordance with \cite[Theorem 1.5]{DFSV} for $p>2$ and with \cite[Theorem 1]{FA2} for $p=2$\,. Moreover, for $p=2$ in \cite[Theorem1]{FA2} it is shown that the result is sharp, in the sense that if the requirement on the parameter $\sigma$ is not fulfilled, then a stable solution exists.}
\end{exe}

\begin{rem}
Example \ref{ex1} shows that the exponent $\alpha$ introduced in \eqref{alfabeto} is sharp in condition {\bf (HP1)}, in order to obtain nonexistence of nontrivial stable solutions of \eqref{eq-1}. In particular, see Remark \ref{rem1}, the exponent $\alpha$ on the radius $R$ is sharp in the weighted volume growth assumption \eqref{2}, for the class of potential functions $V$ satisfying \eqref{1}.

On the other hand, Theorem \ref{thm2} and Corollary \ref{cor1} show that the exponent $-1$ on the logarithm of the radius $R$ in the weighted volume growth condition \eqref{2} is not sharp in general and can be increased, if one strengthens the assumptions on the potential $V$, in
particular imposing hypotheses on its behavior at infinity such as \eqref{29} or \eqref{41}.
\end{rem}

\section{Auxiliary results}\label{ARes}
In this Section we prove the following two propositions, that will have a crucial role in the proof of Theorems \ref{thm1} and \ref{thm2}\,.
\begin{prop}\label{PR:1}
 Let $u\in W^{1,p}_{\rm loc}(M)\cap
L^\sigma_{\rm loc}(M)$ be a \emph{stable} solution of \eqref{EQ:1}
with $\sigma
>(p-1)$ and $p\geq2$. Let also $\bar\gamma$, $\delta_\gamma$, $\gamma$ be as in \eqref{defgamma} and \eqref{defdelta}
and consider $k$ such that
$$
k \geq \max\left\{\frac{p+\gamma}{\sigma-p+1}\,, \, 2\right\}\,.
$$
Then
there exists a positive constant $c=c(p,\sigma,k)$ such that
\begin{equation}\label{EQ:10bis}
\int_M V(x)\,|u|^{\sigma+\gamma} \psi^{pk} \,d\mu\leq
c\,\delta^{-p\frac{\gamma+\sigma}{\sigma-p+1}}
\int_MV(x)^{-\frac{\gamma+p-1}{\sigma-p+1}} |\nabla
\psi|^{p\frac{\sigma+\gamma}{\sigma-p+1}}\,d\mu
\end{equation}
and
\begin{equation}\label{EQ:11tris}
\begin{split}
&\int_M |\nabla u|^p |u|^{\gamma-1} \psi^{pk}\,d\mu \leq
 c\,\delta^{-p\frac{\gamma+\sigma}{\sigma-p+1}}
 \int_MV(x)^{-\frac{\gamma+p-1}{\sigma-p+1}} |\nabla \psi|^{p\frac{\sigma+\gamma}{\sigma-p+1}}\,d\mu
\end{split}
\end{equation}
for all test functions $\psi \in \operatorname{Lip}_c(M)$ with $0\leq \psi \leq 1$.
\end{prop}

\begin{prop}\label{PR:2}
Under the same assumptions of Proposition \ref{PR:1}, for every
$\delta>0$ small enough one has
\begin{equation}\label{EQ:10ter}
\begin{aligned}
&\int_M V(x)\,|u|^{\sigma+\gamma} \psi^{pk} \,d\mu\\
&\quad\leq c\left[\int_{M\setminus K}
V(x)\,|u|^{\sigma+\overline\gamma}\,d\mu\right]^\frac{\gamma+p-1}{\sigma+\overline\gamma}
\left[\delta^{-p\frac{\sigma+\overline\gamma}{\sigma-p+1+\delta}}\int_{M\setminus
K}V(x)^{-\frac{\gamma+p-1}{\sigma-p+1+\delta}} |\nabla
\psi|^{p\frac{\sigma+\overline\gamma}{\sigma-p+1+\delta}}\,d\mu\right]^\frac{\sigma-p+1+\delta}{\sigma+\overline\gamma}
\end{aligned}
\end{equation}
for some positive constant $c$, where $K=\{x\in M\,|\,\psi(x)=1\}$.
\end{prop}

Proposition \ref{PR:1} provides an important estimate on the
integrability of $u$ and $\nabla u$. As we will see, our
nonexistence results will follow by showing that the right-hand
sides vanish, under suitably weighted volume growth assumptions on
geodesic balls or annuli.

\begin{proof}[Proof of Proposition \ref{PR:1}]
\textbf{Step 1}. For any non-negative $\varphi\in
\operatorname{Lip}_c (\Omega)$, by density arguments, we can plug
$$\Phi = |u|^{\gamma-1}u \varphi^p$$ in \eqref{EQ:1} and get
\begin{equation*}
\gamma \int_M |\nabla u|^{p} |u|^{\gamma-1} \varphi^p\,d\mu
\,=\,-p\int_M |\nabla u|^{p-2}(\nabla u,\nabla \varphi) |u|^{\gamma
-1}u\varphi^{p-1}\,d\mu+\int_M
V(x)|u|^{\sigma+\gamma}\varphi^p\,d\mu
\end{equation*}
so that
\begin{equation*}
\gamma\int_M|\nabla u|^{p}|u|^{\gamma-1} \varphi^p\,d\mu\, \leq\, p
\int_M |\nabla u|^{p-1}|\nabla \varphi||u|^\gamma
\varphi^{p-1}\,d\mu+\int_M V(x)|u|^{\sigma+\gamma}\varphi^p\,d\mu
.\end{equation*} Writing
$|u|^\gamma=|u|^{[(\gamma-1)\frac{p-1}{p}+\frac{\gamma +(p-1)}{p}]}$
and exploiting  Young's inequality with exponents $p$ and
$\frac{p}{p-1}$ we have
\begin{align*}
\gamma\int_M|\nabla u|^p |u|^{\gamma-1} \varphi^p\,d\mu \leq&
\epsilon \int_M|\nabla u|^p |u|^{\gamma-1} \varphi^p\,d\mu \\&+
\frac{c(p)}{\varepsilon^{p-1}} \int_M |\nabla \varphi|^p
|u|^{\gamma+p-1}\,d\mu+\int_M V(x)
|u|^{\sigma+\gamma}\varphi^p\,d\mu
\end{align*}
that we rewrite as
\begin{equation}\label{EQ:4}
\begin{aligned}
(\gamma-\varepsilon)\int_M|\nabla u|^p |u|^{\gamma-1}
\varphi^p\,d\mu \leq \frac{c(p)}{\varepsilon^{p-1}} \int_M |\nabla
\varphi|^p |u|^{\gamma+(p-1)}\,d\mu+\int_M V(x)
|u|^{\sigma+\gamma}\varphi^p\,d\mu\,\,.
\end{aligned}
\end{equation}

\noindent \textbf{Step 2}.
Now we exploit the stability of $u$ with the choice of test function
$$\tilde{\Phi}=|u|^\frac{\gamma-1}{2}u\varphi^{\frac{p}{2}}$$
and  we get
\begin{equation}\label{EQ:7bis}
\begin{aligned}
\frac{\sigma}{p-1}\int_{M} V(x)\,|u|^{\sigma+\gamma} \varphi^p \leq&
\left(\frac{1+\gamma}{2}\right)^2 \int_{M} |\nabla
u|^{p}|u|^{\gamma-1}\varphi^p +\frac{p^2}{4}\int_{M}|\nabla
u|^{p-2}|u|^{\gamma+1}|\nabla \varphi|^2
\varphi^{p-2}\\&+p\frac{\gamma+1}{2}\int_{M}|\nabla
u|^{p-1}|u|^{\gamma} \varphi^{p-1} |\nabla \varphi| .
\end{aligned}
\end{equation}
In case $p>2$ we write $|u|^{\gamma +1}=|u|^{[(\gamma
-1)\frac{(p-2)}{p}+2\frac{(\gamma +(p-1))}{p}]}$ and using Young's
inequality with exponents $\frac{p}{p-2}$ and $\frac{p}{2}$ we get
\begin{equation}\label{43}
\begin{aligned}
&\frac{p^2}{4}\int_{M}  |\nabla u|^{p-2}|u|^{\gamma+1}
\varphi^{p-2}|\nabla \varphi|^2 \leq \frac{\epsilon}{2} \int_{M}
|\nabla u|^{p}|u|^{\gamma-1} \varphi^p
+\frac{c(p)}{\varepsilon^{\frac{p-2}{2}}} \int_{M} |\nabla
\varphi|^p |u|^{\gamma+(p-1)} .
\end{aligned}
\end{equation}
Also, by writing
$|u|^{\gamma}=|u|^{(\gamma-1)\frac{p-1}{p}+\frac{\gamma+(p-1)}{p}}$
and using Young's inequality with exponents $\frac{p}{p-1}$ and
 $p$, we obtain
\begin{equation}\label{42}
p\frac{\gamma+1}{2}\int_{M} |\nabla
u|^{p-1}|u|^{\gamma}\varphi^{p-1}|\nabla \varphi| \,\leq\,
\frac{\epsilon}{2} \int_{M} |\nabla u|^p \varphi^p |u|^{\gamma-1}
+\frac{c(p)}{\varepsilon^{p-1}} \int_{M} |\nabla \varphi|^p
|u|^{\gamma+(p-1)}.
\end{equation}
Recollecting and exploiting  \eqref{EQ:7bis}, we get
\begin{equation}\label{EQ:8bis}
\begin{split}
\frac{\sigma}{p-1}\int_{M} V(x)\,|u|^{\sigma+\gamma} \varphi^p
\,\leq\,&
\left[\left(\frac{\gamma+1}{2}\right)^2+\epsilon\right]\int_{M}|\nabla
u|^p |u|^{\gamma-1} \varphi^p+\\
&+c(p)\left
(\frac{1}{\varepsilon^{\frac{p-2}{2}}}+\frac{1}{\varepsilon^{p-1}}
\right )\int_{M} |\nabla \varphi|^p |u|^{\gamma+(p-1)} .\end{split}
\end{equation}
By \eqref{EQ:4} and \eqref{EQ:8bis},
\begin{equation}\label{44}
\begin{aligned}
\left(\frac{\sigma}{p-1}\right)\int_{M}V(x)\,
|u|^{\sigma+\gamma}\varphi^p \leq&
\left[\left(\frac{\gamma+1}{2}\right)^2
+\epsilon\right]\frac{1}{(\gamma-\epsilon)}
\int_{M}V(x)\, |u|^{\sigma+\gamma}\varphi^p\\
&\qquad+c(p,\gamma)\left
(\frac{1}{\varepsilon^{\frac{p-2}{2}}}+\frac{1}{\varepsilon^{p-1}}
\right ) \int_{M} |\nabla \varphi|^p |u|^{\gamma+(p-1)}\,.
\end{aligned}
\end{equation}

In case $p=2$ one does not need inequality \eqref{43}. Indeed, starting from \eqref{EQ:7bis} and using \eqref{EQ:4} and \eqref{42} one can easily see that \eqref{44} holds also when $p=2$.

Since
 $c(p,\gamma)$ is bounded away from zero and from infinity in the range of $\gamma$ that we are considering, from now on we omit the dependence on $\gamma$.
Setting
\begin{equation}\label{9bis} \omega=\omega_\epsilon
=\frac{\sigma}{p-1}-\left(\frac{(\gamma+1)^2}{4}
+\epsilon\right)\frac{1}{\gamma-\epsilon}
\end{equation}
and considering with no loss of generality bounded values of
$\varepsilon$, e.g. $0<\varepsilon\leq2^{-1}$, we arrive at
\begin{equation}\label{EQ:9bis}
\omega \int_{M} V(x)\,|u|^{\sigma+\gamma} \varphi^{p} \leq
\frac{c(p)}{\varepsilon^{p-1}} \int_{M} |\nabla \varphi|^{p}
|u|^{\gamma+(p-1)}
\end{equation}
for any non-negative $\varphi\in \operatorname{Lip}_c (M)$. It is
now convenient to set
\[
g(\gamma)\,:=\,\frac{\sigma}{p-1}-\frac{(\gamma+1)^2}{4\gamma}.
\]
By direct computation it follows that
 $g(\bar\gamma)=0$ and $g'(\bar\gamma)<0$ for $\gamma>1$. Therefore, for the range of $\delta$ we are considering, recalling that $\delta=\bar\gamma-\gamma$, we infer that
 $g'(\gamma)$ is bounded away from zero.
  Now,  by the mean value theorem, we deduce that
\begin{equation}\nonumber
\begin{split}
\omega_\epsilon &=\frac{\sigma}{p-1}-\left(\frac{(\gamma+1)^2}{4}
+\epsilon\right)\frac{1}{\gamma-\epsilon}\\
&=-(g(\bar\gamma)-g(\gamma))-\frac{(\gamma+1)^2}{4}\cdot\frac{\varepsilon}{(\gamma-\varepsilon)\gamma}-\frac{\varepsilon}{(\gamma-\varepsilon)}\\
&\geq 2\tilde c(p,\sigma)\delta-\varepsilon \left(   \frac{(\gamma+1)^2}{2\gamma}+2 \right)\,,
\end{split}
\end{equation}
where we exploited  the fact that we are assuming e.g.  $0<\varepsilon\leq2^{-1}$ and by assumption $\gamma>1$, so that $\gamma-\varepsilon\geq 2^{-1}$.
Taking into account \eqref{defgamma}, we deduce that
\begin{equation}\nonumber
\omega_\varepsilon >\tilde c(p,\sigma)\delta\qquad\text{for }\quad
\varepsilon =\tilde c(p,\sigma)\cdot \left(
\frac{(\gamma+1)^2}{2\gamma}+2  \right)^{-1} \delta\,.
\end{equation}

Therefore \eqref{EQ:9bis} can be rewritten as
\begin{equation}\label{EQ:9bis xvxv}
 \int_{M} V(x)\,|u|^{\sigma+\gamma}
\varphi^{p} \leq \frac{c(p,\sigma)}{\delta^p} \int_{M} |\nabla
\varphi|^{p} |u|^{\gamma+(p-1)}\,.
\end{equation}
Furthermore, possibly relabeling the constant,  \eqref{EQ:4}
provides
\begin{eqnarray}\label{EQ:4vvxv}
\int_{M}|\nabla u|^p |u|^{\gamma-1} \varphi^p &\leq&
\frac{c(p,\sigma)}{\delta^p}\int_{M} |\nabla \varphi|^{p}
|u|^{\gamma+(p-1)}\,.
\end{eqnarray}

\noindent\textbf{Step 3}. Let us now consider $$\psi\in
\operatorname{Lip}_c(M)$$
 with $0\leq\psi\leq1$ and take $\varphi=\psi^k$ in \eqref{EQ:9bis xvxv}. Then we obtain
\begin{equation}\label{EQ:7}
\int_{M}V(x)\, |u|^{\sigma+\gamma} \psi^{pk}\,d\mu \leq
\frac{c(p,\sigma,k)}{\delta^p}\int_{M} |u|^{\gamma+(p-1)}
\psi^{p(k-1)} |\nabla \psi|^p\,d\mu.
\end{equation}
Hence
\begin{equation*}
\begin{aligned}
\int_{M}V(x)\, |u|^{\sigma+\gamma}
\psi^{pk}\,d\mu\leq&\frac{c(p,\sigma,k)}{\delta^p} \left[\int_{M}
\left(V(x)^{\frac{\gamma+(p-1)}{\sigma+\gamma}}|u|^{\gamma+(p-1)}\psi^{p(k-1)}\right)^\frac{\sigma+\gamma}{\gamma+(p-1)}\,d\mu\right]^\frac{\gamma+(p-1)}
{\sigma+\gamma}\\&\hspace{2cm}\times\left[\int_{M}
\left(V(x)^{-\frac{\gamma+(p-1)}{\sigma+\gamma}}|\nabla
\psi|^p\right)^\frac{\sigma+\gamma}{\sigma-(p-1)}\,d\mu\right]^\frac{\sigma-(p-1)}{\sigma+\gamma},
\end{aligned}
\end{equation*}
where we used the fact that $\sigma >( p-1)$. Furthermore
\begin{equation}\label{16bis}
p(k-1)\left(\frac{\sigma+\gamma}{\gamma+(p-1)}\right) \geq p\,k
\end{equation}
since we assumed $k \geq \frac{\sigma+\gamma}{\sigma-(p-1)}$.
Consequently, recalling that $0\leq \psi \leq 1$, we have
\begin{equation*}
\begin{aligned}
&\int_{M} V(x)\,|u|^{\sigma+\gamma} \psi^{pk} \leq\\
&\qquad\frac{c(p,\sigma,k)}{\delta^p} \left[\int_{M}
V(x)\,|u|^{\sigma+\gamma}\psi^{pk}\right]^\frac{\gamma+(p-1)}{\sigma+\gamma}
\left[\int_{M}V(x)^{-\frac{\gamma+(p-1)}{\sigma-(p-1)}} |\nabla
\psi|^{p\frac{\sigma+\gamma}{\sigma-(p-1)}}\right]^\frac{\sigma-(p-1)}{\sigma+\gamma}.
\end{aligned}
\end{equation*}
Hence
\begin{equation}\nonumber
\int_{M} V(x)\,|u|^{\sigma+\gamma} \psi^{pk} \leq
c(p,\sigma,k)\delta^{-p\frac{\sigma+\gamma}{\sigma-p+1}}
\int_{M}V(x)^{-\frac{\gamma+(p-1)}{\sigma-(p-1)}} |\nabla
\psi|^{p\frac{\sigma+\gamma}{\sigma-(p-1)}}\,.
\end{equation}
that is \eqref{EQ:10bis}.

\noindent \textbf{Step 4}. Again we consider $\psi\in
\operatorname{Lip}_c(M)$ such that $0\leq\psi\leq 1$. Then we
evaluate \eqref{EQ:4vvxv} for $\varphi=\psi^k$, obtaining
\begin{equation*}
\begin{split}
&\int_{M} |\nabla u|^p |u|^{\gamma-1} \psi^{pk} \leq
\frac{c(p,\sigma,k)}{\delta^p}  \int_{M}
|u|^{\gamma+(p-1)} \psi^{p(k-1)} |\nabla \psi|^{p} \\
&\qquad\leq  \frac{c(p,\sigma,k)}{\delta^p} \left[\int_{M}
\left(V(x)^{\frac{\gamma+(p-1)}{\sigma+\gamma}}|u|^{\gamma+(p-1)}\psi^{p(k-1)}\right)^\frac{\sigma+\gamma}{\gamma+(p-1)}\right]^\frac{\gamma+(p-1)}
{\sigma+\gamma}\\
&\qquad\hspace{5cm}
\times\left[\int_{M} \left(V(x)^{-\frac{\gamma+(p-1)}{\sigma+\gamma}}|\nabla \psi|^p\right)^\frac{\sigma+\gamma}{\sigma-(p-1)}\right]^\frac{\sigma-(p-1)}{\sigma+\gamma}\\
 &\qquad\leq \frac{c(p,\sigma,k)}{\delta^p} \left[\int_{M} \left(V(x)|u|^{\sigma+\gamma}\psi^{pk}\right)\right]
 ^\frac{\gamma+(p-1)}{\sigma+\gamma}
\left[\int_{M}
\left(V(x)^{-\frac{\gamma+(p-1)}{\sigma+\gamma}}|\nabla
\psi|^p\right)^\frac{\sigma+\gamma}{\sigma-(p-1)}\right]^\frac{\sigma-(p-1)}{\sigma+\gamma}
\end{split}
\end{equation*}
and so,
taking into account \eqref{EQ:10bis}, we  get \eqref{EQ:11tris}, namely
\begin{equation}\nonumber
\begin{split}
&\int_{M} |\nabla u|^p |u|^{\gamma-1} \psi^{pk} \leq
 \frac{c(p,\sigma,k)}{\delta^p} \left[\frac{c(p,\sigma,k)}{\delta^{p\frac{\sigma+\gamma}{\sigma-p+1}}}\right]
 ^\frac{\gamma+(p-1)}{\sigma+\gamma}
 \int_{M}V(x)^{-\frac{\gamma+(p-1)}{\sigma-(p-1)}} |\nabla \psi|^{p\frac{\sigma+\gamma}{\sigma-(p-1)}}\,.
\end{split}
\end{equation}
\end{proof}

\begin{proof}[Proof of Proposition \ref{PR:2}]
Note that, since we are under the same assumptions of Proposition
\ref{PR:1}, \eqref{EQ:7} holds. Since $\nabla\psi=0$ a.e. in $K$ we
have
\begin{equation*}
\int_{M}V(x)\, |u|^{\sigma+\gamma} \psi^{pk}\,d\mu \leq
\frac{c}{\delta^p}\int_{M\setminus K} |u|^{\gamma+(p-1)}
\psi^{p(k-1)} |\nabla \psi|^p\,d\mu.
\end{equation*}

We now apply H\"{o}lder's inequality to obtain
\begin{equation*}
\begin{aligned}
&\int_{M}V(x)\, |u|^{\sigma+\gamma} \psi^{pk}\,d\mu \\
&\quad\leq \frac{c}{\delta^p}\left[\int_{M\setminus K}V
|u|^{\sigma+\overline\gamma}
\psi^{p(k-1)\frac{\sigma+\overline\gamma}{\gamma+p-1}}\,d\mu\right]^\frac{\gamma+p-1}{\sigma+\overline\gamma}
\left[\int_{M\setminus K}
|\nabla\psi|^{p\frac{\sigma+\overline\gamma}{\sigma-p+1+\delta}}V^{-\frac{\gamma+p-1}{\sigma-p+1+\delta}}\,d\mu\right]^\frac{\sigma-p+1+\delta}{\sigma+\overline\gamma}\\
&\quad\leq c\left[\int_{M\setminus K}V
|u|^{\sigma+\overline\gamma}\,d\mu\right]^\frac{\gamma+p-1}{\sigma+\overline\gamma}
\left[\delta^{-p\frac{\sigma+\overline\gamma}{\sigma-p+1+\delta}}\int_{M\setminus
K}
|\nabla\psi|^{p\frac{\sigma+\overline\gamma}{\sigma-p+1+\delta}}V^{-\frac{\gamma+p-1}{\sigma-p+1+\delta}}\,d\mu\right]^\frac{\sigma-p+1+\delta}{\sigma+\overline\gamma},
\end{aligned}
\end{equation*}
since $0\leq\psi\leq1$, that is the conclusion.
\end{proof}

\medskip

\section{Proof of Theorems \ref{thm1} and \ref{thm2}}\label{PrMRes}

\begin{proof}[Proof of Theorem \ref{thm1}] We start by considering the case when there exist positive
constants $C$, $C_0$ such that for every small enough $\epsilon>0$
and every large enough $R>0$ one has
\begin{equation}\label{23}
\int_{B_R\setminus B_{R/2}}V^{-\beta+\epsilon}\,d\mu \leq
CR^{\alpha+C_0\epsilon}(\log R)^{b},
\end{equation}
for some $b<-1$.

We use inequality \eqref{EQ:10bis} with the particular choice of
test function
\begin{equation}\label{22}
\psi(x)\equiv\psi_n(x)=\varphi(x)\eta_n(x),
\end{equation}
where
\begin{equation}\label{8}
 \varphi(x)=\begin{cases}\begin{array}{ll}
   1&\quad\text{for }r(x)<R,\\
   \left(\frac{r(x)}{R}\right)^{-C_1\delta}&\quad\text{for }r(x)\geq
   R,
   \end{array}\end{cases}
\end{equation}
with $R>0$ fixed, $C_1>0$ a fixed constant to be chosen later large
enough and $\delta=\overline\gamma-\gamma>0$ as defined in
\eqref{defgamma} and \eqref{defdelta}, and for $n\in\N$
\begin{equation}\label{9}
 \eta_n(x)=\begin{cases}\begin{array}{ll}
   1&\quad\text{for }r(x)<nR,\\
   2-\frac{r(x)}{nR}&\quad\text{for }nR\leq r(x)\leq 2nR,\\
   0&\quad\text{for }r(x)\geq2nR.
 \end{array}\end{cases}
\end{equation}
We further specialize the choice of $\delta$, and hence of $\gamma$,
in inequality \eqref{EQ:10bis} by setting $\delta=\frac{1}{\log R}$.
Here and for the rest of the proof we always assume that $R>0$ is
chosen large enough so that condition \eqref{volHP2} holds.

Note that $\psi_n\in Lip_c(M)$ with $0\leq\psi_n\leq1$ on $M$ and
one has
$$\nabla\psi_n=\varphi\nabla\eta_n+\eta_n\nabla\varphi.$$ Hence for
every $a>0$ one has $$|\nabla\psi_n|^a\leq
2^a(\varphi^a|\nabla\eta_n|^a+\eta_n^a|\nabla\varphi|^a).$$
Substituting in \eqref{EQ:10bis} one obtains
\begin{equation}\label{15}
\int_M V\,|u|^{\sigma+\gamma} \psi_n^{pk} \,d\mu\leq
C\,\delta^{-p\frac{\gamma+\sigma}{\sigma-p+1}}[I_1+I_2],
\end{equation}
where
\begin{equation}\label{33}
\begin{aligned}
I_1&=&\int_MV^{-\frac{\gamma+p-1}{\sigma-p+1}}\varphi^{p\frac{\sigma+\gamma}{\sigma-p+1}}
|\nabla \eta_n|^{p\frac{\sigma+\gamma}{\sigma-p+1}}\,d\mu,\\
I_2&=&\int_MV^{-\frac{\gamma+p-1}{\sigma-p+1}}\eta_n^{p\frac{\sigma+\gamma}{\sigma-p+1}}
|\nabla \varphi|^{p\frac{\sigma+\gamma}{\sigma-p+1}}\,d\mu.
\end{aligned}
\end{equation}

Now we easily see that
\begin{equation}\label{34}
\begin{aligned}
\delta^{-p\frac{\gamma+\sigma}{\sigma-p+1}}I_1&=\delta^{-p\frac{\gamma+\sigma}{\sigma-p+1}}\int_MV^{-\frac{\gamma+p-1}{\sigma-p+1}}\varphi^{p\frac{\sigma+\gamma}{\sigma-p+1}}
|\nabla \eta_n|^{p\frac{\sigma+\gamma}{\sigma-p+1}}\,d\mu,\\
&=\delta^{-p\frac{\gamma+\sigma}{\sigma-p+1}}\int_{B_{2nR}\setminus
B_{nR}}V^{-\frac{\overline\gamma+p-1-\delta}{\sigma-p+1}}\varphi^{p\frac{\sigma+\gamma}{\sigma-p+1}}
|\nabla \eta_n|^{p\frac{\sigma+\gamma}{\sigma-p+1}}\,d\mu,\\
&\leq\delta^{-p\frac{\gamma+\sigma}{\sigma-p+1}}n^{-pC_1\delta\frac{\sigma+\gamma}{\sigma-p+1}}(nR)^{-p\frac{\sigma+\gamma}{\sigma-p+1}}\int_{B_{2nR}\setminus
B_{nR}}V^{-\frac{\overline\gamma+p-1-\delta}{\sigma-p+1}}\,d\mu.
\end{aligned}
\end{equation}
Hence, for $R>0$ large enough and thus $\delta>0$ small enough, by
condition \eqref{23} we have
\begin{equation*}
\delta^{-p\frac{\gamma+\sigma}{\sigma-p+1}}I_1\leq
C\delta^{-p\frac{\gamma+\sigma}{\sigma-p+1}}n^{-pC_1\delta\frac{\sigma+\gamma}{\sigma-p+1}}(nR)^{-p\frac{\sigma+\gamma}{\sigma-p+1}}
(2nR)^{\alpha+\frac{C_0\delta}{\sigma-p+1}}\big(\log(2nR)\big)^b.
\end{equation*}
By our definition of $\alpha$, the previous inequality yields
\begin{equation}\label{10}
\delta^{-p\frac{\gamma+\sigma}{\sigma-p+1}}I_1\leq
C\delta^{-p\frac{\overline\gamma+\sigma-\delta}{\sigma-p+1}}n^{\frac{\delta}{\sigma-p+1}(C_0+p-pC_1(\sigma+\overline\gamma-\delta))}(R^\delta)^{\frac{p+C_0}{\sigma-p+1}}
\big(\log(2nR)\big)^b.
\end{equation}
We choose $R>0$ so large that $\delta=\frac{1}{\log
R}<\frac{\sigma+\overline\gamma}{2}$ and
\begin{equation}\label{13}
C_1>2\frac{C_0+p+1}{p(\sigma+\overline\gamma)}
\end{equation}
in \eqref{8}. Since $R^\delta= e$, from \eqref{10} we deduce
\begin{equation}\label{11}
\delta^{-p\frac{\gamma+\sigma}{\sigma-p+1}}I_1\leq C
\delta^{-p\frac{\overline\gamma+\sigma-\delta}{\sigma-p+1}}n^{-\frac{\delta}{\sigma-p+1}}
\big(\log(2nR)\big)^b.
\end{equation}
Now we estimate
\begin{equation}\label{12}
\begin{aligned}
\delta^{-p\frac{\gamma+\sigma}{\sigma-p+1}}I_2&=\delta^{-p\frac{\gamma+\sigma}{\sigma-p+1}}\int_MV^{-\frac{\gamma+p-1}{\sigma-p+1}}\eta_n^{p\frac{\sigma+\gamma}{\sigma-p+1}}
|\nabla\varphi|^{p\frac{\sigma+\gamma}{\sigma-p+1}}\,d\mu\\
&\leq\delta^{-p\frac{\gamma+\sigma}{\sigma-p+1}}\int_{M\setminus
B_R}
V^{-\frac{\gamma+p-1}{\sigma-p+1}}\big(C_1\delta R^{C_1\delta}r(x)^{-C_1\delta-1}\big)^{p\frac{\sigma+\gamma}{\sigma-p+1}}\,d\mu\\
&\leq
CR^{C_1p\delta\frac{\sigma+\gamma}{\sigma-p+1}}\int_{M\setminus B_R}
V^{-\frac{\overline\gamma+p-1-\delta}{\sigma-p+1}}r(x)^{-(C_1\delta+1)p\frac{\sigma+\gamma}{\sigma-p+1}}\,d\mu.
\end{aligned}
\end{equation}
In order to proceed with our estimates we recall that if
$f:[0,\infty)\rightarrow[0,\infty)$ is a nonnegative decreasing
function and \eqref{23} holds, then for any small enough
$\epsilon>0$ and any sufficiently large $R>1$ we have
\begin{equation}\label{219}
\int_{M\setminus B_R} f(r(x)) V(x)^{-\beta+\epsilon}\,d\mu\leq
C\int_{\frac{R}{2}}^{+\infty}f(r)r^{\alpha+C_0\epsilon-1}(\log
r)^b\,dr
\end{equation}
for some positive constant $C$, see \cite[formula (2.19)]{GrigS}.
Recalling that $R^\delta=e$ and using inequality \eqref{219} in
\eqref{12}, we obtain
\begin{equation*}
\delta^{-p\frac{\gamma+\sigma}{\sigma-p+1}}I_2\leq
C\int_{\frac{R}{2}}^{+\infty}r^{-p\frac{\sigma+\gamma}{\sigma-p+1}(C_1\delta+1)+\alpha+C_0\frac{\delta}{\sigma-p+1}-1}(\log
r)^b\,dr
\end{equation*}
Now we define
\begin{equation}\label{38}
s:=p\frac{\sigma+\gamma}{\sigma-p+1}(C_1\delta+1)-\alpha-C_0\frac{\delta}{\sigma-p+1}=\frac{\delta}{\sigma-p+1}\big((\sigma+\overline\gamma-\delta)pC_1-(p+C_0)\big)
\end{equation}
and by our assumptions on $C_1$ and $\delta$, see also \eqref{13},
we have
\begin{equation}\label{39}
0<\frac{\delta}{\sigma-p+1}\leq s \leq
pC_1\frac{(\sigma+\overline\gamma)}{\sigma-p+1}\delta.
\end{equation}
Thus
\begin{equation*}
\delta^{-p\frac{\gamma+\sigma}{\sigma-p+1}}I_2\leq
C\int_{\frac{R}{2}}^{+\infty}r^{-s}(\log r)^b\,\frac{dr}{r}
\end{equation*}
and the change of variable $\xi=s\log r$ yields
\begin{equation}\label{14}
\delta^{-p\frac{\gamma+\sigma}{\sigma-p+1}}I_2\leq
Cs^{-b-1}\int_{1}^{+\infty}e^{-\xi}\xi^b\,d\xi\leq C\delta^{-b-1}.
\end{equation}
Now we insert inequalities \eqref{11} and \eqref{14} into \eqref{15}
and we obtain
\begin{equation}\label{17}
 \int_{B_R} V\,|u|^{\sigma+\gamma}\,d\mu\leq \int_M
V\,|u|^{\sigma+\gamma} \psi_n^{pk} \,d\mu\leq
C\big[\delta^{-p\frac{\overline\gamma+\sigma-\delta}{\sigma-p+1}}n^{-\frac{\delta}{\sigma-p+1}}
\big(\log(2nR)\big)^b +\delta^{-b-1}\big].
\end{equation}
Passing to the $\liminf$ as $n$ tends to infinity in the previous
inequality we deduce
\[ \int_{B_R} V\,|u|^{\sigma+\overline\gamma-\frac{1}{\log R}}\,d\mu\leq C(\log R)^{b+1}.\]
Then, passing to the $\liminf$ as $R$ tends to infinity and using
Fatou's Lemma, we obtain
\[\int_M V\,|u|^{\sigma+\overline\gamma}\,d\mu=0,\]
which immediately yields $u=0$ a.e. in $M$, since $V>0$ a.e. in $M$.

We now turn to the case when {\bf (HP1)} holds, i.e. when we have
\eqref{23} with $b=-1$. We use again Proposition \ref{PR:1}, with
the test functions $\psi_n=\varphi\eta_n$ as in \eqref{22}. From
formulas \eqref{11}, \eqref{14} and \eqref{17} in the proof of the
previous case, but now with $b=-1$ instead of $b<-1$, for every
$R>0$ large enough and every $n\in\N$ we have
\begin{equation}\label{21}
\int_{B_R} V\,|u|^{\sigma+\gamma}\,d\mu\leq
C\big[\delta^{-p\frac{\overline\gamma+\sigma-\delta}{\sigma-p+1}}n^{-\frac{\delta}{\sigma-p+1}}
\big(\log(2nR)\big)^{-1} +1\big].
\end{equation}
Passing to the $\liminf$ as $n$ and $R$ tend to infinity in
\eqref{21}, due to Fatou's Lemma we obtain
\begin{equation}\label{18}
\int_{M} V\,|u|^{\sigma+\overline\gamma}\,d\mu\leq C.
\end{equation}
Now we use Proposition \ref{PR:2} with the family of test functions
$\psi_n=\varphi\eta_n$, and from \eqref{EQ:10ter} we deduce
\begin{equation}\label{28}
\int_M V(x)\,|u|^{\sigma+\gamma} \psi_n^{pk} \,d\mu\leq
C\left[\int_{M\setminus B_R}
V(x)\,|u|^{\sigma+\overline\gamma}\,d\mu\right]^\frac{\gamma+p-1}{\sigma+\overline\gamma}
\left[\delta^{-p\frac{\sigma+\overline\gamma}{\sigma-p+1+\delta}}(I_3+I_4)\right]^\frac{\sigma-p+1+\delta}{\sigma+\overline\gamma}
\end{equation}
where we set
\begin{equation*}
\begin{aligned}
I_3&=&\int_MV^{-\frac{\gamma+p-1}{\sigma-p+1+\delta}}\varphi^{p\frac{\sigma+\overline\gamma}{\sigma-p+1+\delta}}
|\nabla \eta_n|^{p\frac{\sigma+\overline\gamma}{\sigma-p+1+\delta}}\,d\mu,\\
I_4&=&\int_MV^{-\frac{\gamma+p-1}{\sigma-p+1+\delta}}\eta_n^{p\frac{\sigma+\overline\gamma}{\sigma-p+1+\delta}}
|\nabla
\varphi|^{p\frac{\sigma+\overline\gamma}{\sigma-p+1+\delta}}\,d\mu.
\end{aligned}
\end{equation*}
Now we proceed to estimate $I_3$, $I_4$ as $I_1$ and $I_2$,
respectively. Indeed, we start noting that
$$-\frac{\gamma+p-1}{\sigma-p+1+\delta}=-\beta+\frac{\sigma+\overline\gamma}{(\sigma-p+1)(\sigma-p+1+\delta)}\delta.$$
Thus we easily obtain
\begin{equation*}
\delta^{-p\frac{\sigma+\overline\gamma}{\sigma-p+1+\delta}}I_3\leq\delta^{-p\frac{\sigma+\overline\gamma}{\sigma-p+1+\delta}}
n^{-pC_1\delta\frac{\sigma+\overline\gamma}{\sigma-p+1+\delta}}(nR)^{-p\frac{\sigma+\overline\gamma}{\sigma-p+1+\delta}}\int_{B_{2nR}\setminus
B_{nR}}V^{-\beta+\frac{\sigma+\overline\gamma}{(\sigma-p+1)(\sigma-p+1+\delta)}\delta}\,d\mu.
\end{equation*}
By condition \eqref{volHP2} for $\delta>0$ small enough we have
\begin{equation*}
\delta^{-p\frac{\sigma+\overline\gamma}{\sigma-p+1+\delta}}I_3\leq
C\delta^{-p\frac{\sigma+\overline\gamma}{\sigma-p+1+\delta}}
n^{\frac{(\sigma+\overline\gamma)(C_0+p-C_1p(\sigma-p+1))}{(\sigma-p+1)(\sigma-p+1+\delta)}\delta}
R^{\frac{(\sigma+\overline\gamma)(p+C_0)}{(\sigma-p+1+)(\sigma-p+1+\delta)}\delta}(\log(2nR))^{-1}.
\end{equation*}
Choosing $\delta>0$ small enough and
\begin{equation}\label{25}
C_1>\frac{2+p+C_0}{(\sigma-p+1)p}
\end{equation}
we obtain
\begin{equation}\label{24}
\delta^{-p\frac{\sigma+\overline\gamma}{\sigma-p+1+\delta}}I_3\leq
C\delta^{-p\frac{\sigma+\overline\gamma}{\sigma-p+1+\delta}}
n^{-\frac{\sigma+\overline\gamma}{(\sigma-p+1)^2}\delta}
(\log(2nR))^{-1}.
\end{equation}
On the other hand we have
\begin{equation*}
\delta^{-p\frac{\sigma+\overline\gamma}{\sigma-p+1+\delta}}I_4\leq
CR^{C_1p\delta\frac{\sigma+\overline\gamma}{\sigma-p+1+\delta}}\int_{M\setminus
B_R}
V^{-\beta+\frac{\sigma+\overline\gamma}{(\sigma-p+1)(\sigma-p+1+\delta)}\delta}r(x)^{-(C_1\delta+1)p\frac{\sigma+\overline\gamma}{\sigma-p+1+\delta}}\,d\mu.
\end{equation*}
By inequality \eqref{219} we obtain
\begin{equation}\label{26}
\delta^{-p\frac{\sigma+\overline\gamma}{\sigma-p+1+\delta}}I_4\leq
C\int_{\frac{R}{2}}^{+\infty}
r^{-p\frac{\sigma+\overline\gamma}{\sigma-p+1+\delta}(C_1\delta+1)+\alpha+C_0\delta\frac{\sigma+\overline\gamma}{(\sigma-p+1)(\sigma-p+1+\delta)}-1}(\log
r)^{-1}\,dr,
\end{equation}
and we set
\[
t:=p\frac{\sigma+\overline\gamma}{\sigma-p+1+\delta}(C_1\delta+1)-\alpha-C_0\delta\frac{\sigma+\overline\gamma}{(\sigma-p+1)(\sigma-p+1+\delta)}.
\]
By \eqref{25} and for $\delta>0$ small enough we have
\[
0<\frac{\sigma+\overline\gamma}{(\sigma-p+1)^2}\delta<t<2\frac{\sigma+\overline\gamma}{(\sigma-p+1)^2}\delta
\]
and using the change of variables $\xi=t\log r$ in \eqref{26} we
obtain
\begin{equation}\label{27}
\delta^{-p\frac{\sigma+\overline\gamma}{\sigma-p+1+\delta}}I_4\leq
C\int_{1}^{+\infty}e^{-\xi}\xi^{-1}\,d\xi\leq C.
\end{equation}
Hence from \eqref{28}, \eqref{24} and \eqref{27} we deduce that for
every $n\in\N$ and every $R>0$ large enough
\begin{equation*}
\begin{aligned}
&\int_{B_R}V|u|^{\sigma+\overline\gamma-\delta}\,d\mu\leq\int_{M}V|u|^{\sigma+\gamma}\psi_n^{pk}\,d\mu\\
&\qquad\leq C\left[\int_{M\setminus B_R}
V(x)\,|u|^{\sigma+\overline\gamma}\,d\mu\right]^\frac{\gamma+p-1}{\sigma+\overline\gamma}
\left[\delta^{-p\frac{\sigma+\overline\gamma}{\sigma-p+1+\delta}}
n^{-\frac{\sigma+\overline\gamma}{(\sigma-p+1)^2}\delta}
(\log(2nR))^{-1}+1\right]^\frac{\sigma-p+1+\delta}{\sigma+\overline\gamma}.
\end{aligned}
\end{equation*}
Passing to the $\liminf$ as $n$ tends to infinity we have
\begin{equation*}
\int_{B_R}V|u|^{\sigma+\overline\gamma-\delta}\,d\mu\leq
C\left[\int_{M\setminus B_R}
V(x)\,|u|^{\sigma+\overline\gamma}\,d\mu\right]^\frac{\overline\gamma+p-1-\delta}{\sigma+\overline\gamma}
\end{equation*}
and taking the $\liminf$ as $R$ tends to infinity, using \eqref{18}
and Fatou's Lemma we finally infer that
\[\int_MV|u|^{\sigma+\overline\gamma}\,d\mu=0,\]
and thus $u=0$ a.e. in $M$.
\end{proof}

\begin{proof}[Proof of Theorem \ref{thm2}]
The proof of Theorem \ref{thm2} follows along the same lines of
that of Theorem \ref{thm1} in the simpler case when $b<-1$. Indeed, we apply Proposition \ref{PR:1}
with the same choice of test functions $\psi_n=\varphi\eta_n$
introduced in \eqref{22}, and we obtain as in \eqref{15}
\begin{equation}\label{32}
\int_M V\,|u|^{\sigma+\gamma} \psi_n^{pk} \,d\mu\leq
C\,\delta^{-p\frac{\gamma+\sigma}{\sigma-p+1}}[I_1+I_2],
\end{equation}
with $I_1$, $I_2$ defined in \eqref{33}. As in \eqref{34} we deduce
\begin{equation*}
\delta^{-p\frac{\gamma+\sigma}{\sigma-p+1}}I_1\leq\delta^{-p\frac{\gamma+\sigma}{\sigma-p+1}}n^{-pC_1\delta\frac{\sigma+\gamma}{\sigma-p+1}}(nR)^{-p\frac{\sigma+\gamma}{\sigma-p+1}}\int_{B_{2nR}\setminus
B_{nR}}V^{-\frac{\overline\gamma+p-1-\delta}{\sigma-p+1}}\,d\mu.
\end{equation*}
Now by \eqref{31} we have, similarly to \eqref{11},
\begin{equation}\label{35}
\delta^{-p\frac{\gamma+\sigma}{\sigma-p+1}}I_1\leq C
\delta^{-p\frac{\overline\gamma+\sigma-\delta}{\sigma-p+1}}n^{-\frac{\delta}{\sigma-p+1}}
\big(\log(2nR)\big)^b e^{-\frac{\theta\delta}{\sigma-p+1}\log (nR)\log\log(nR)}.
\end{equation}
In order to estimate $I_2$, we note that as in \eqref{12} one has
\begin{equation}\label{36}
\delta^{-p\frac{\gamma+\sigma}{\sigma-p+1}}I_2\leq
CR^{C_1p\delta\frac{\sigma+\gamma}{\sigma-p+1}}\int_{M\setminus B_R}
V^{-\frac{\overline\gamma+p-1-\delta}{\sigma-p+1}}r(x)^{-(C_1\delta+1)p\frac{\sigma+\gamma}{\sigma-p+1}}\,d\mu.
\end{equation}
We now claim that if $f:[0,\infty)\rightarrow[0,\infty)$ is a
nonnegative decreasing function and \eqref{31} holds, then for any
small enough $\epsilon>0$ and any sufficiently large $R>0$ we have
\begin{equation}\label{37}
\int_{M\setminus B_R} f(r(x)) V(x)^{-\beta+\epsilon}\,d\mu\leq
C\int_{\frac{R}{2}}^{+\infty}f(r)r^{\alpha+C_0\epsilon-1}(\log
r)^be^{-\epsilon\theta\log \frac{r}{2}\log\log \frac{r}{2}}\,dr
\end{equation}
for some positive fixed constant $C$. The proof of the claim is
similar to that of \cite[formula (2.19)]{GrigS} and of \cite[formula
(3.25)]{MaMoPu1}. We sketch it here for the reader's convenience. By
the monotonicity of the involved functions, using condition
\eqref{31} we obtain
\begin{align*}
\int_{M\setminus B_R} f(r(x))& V(x)^{-\beta+\epsilon}\,d\mu\\
   &=\sum_{i=0}^{+\infty}\int_{B_{2^{i+1}R}\setminus B_{2^iR}}f(r(x))V(x)^{-\beta+\epsilon}\,d\mu\\
   &\leq\sum_{i=0}^{+\infty}f(2^iR)\int_{B_{2^{i+1}R}\setminus B_{2^iR}}V^{-\beta+\epsilon}\,d\mu\\
   &\leq C\sum_{i=0}^{+\infty}f(2^iR)e^{-\epsilon\theta \log(2^{i}R)\log\log(2^{i}R)}(2^{i+1}R)^{\alpha+C_0\epsilon}(\log(2^{i+1}R))^b\\
   &\leq \hat{C}\sum_{i=0}^{+\infty}f(2^iR)e^{-\epsilon\theta \log(2^{i-1}R)\log\log(2^{i-1}R)}(2^{i-1}R)^{\alpha+C_0\epsilon}(\log(2^{i-1}R))^b\\
   &\leq \hat{C}\sum_{i=0}^{+\infty}\int_{2^{i-1}R}^{2^iR}f(r)e^{-\epsilon\theta\log\frac{r}{2}\log\log\frac{r}{2}}r^{\alpha+C_0\epsilon-1}(\log r)^b\,dr\\
   &= \hat{C}\int_{\frac{R}{2}}^{+\infty}f(r)e^{-\epsilon\theta\log \frac{r}{2}\log\log \frac{r}{2}}r^{\alpha+C_0\epsilon-1}(\log r)^b\,dr.
\end{align*}
Now from \eqref{36} and \eqref{37} it follows that
\begin{equation*}
\delta^{-p\frac{\gamma+\sigma}{\sigma-p+1}}I_2\leq
C\int_{\frac{R}{2}}^{+\infty}r^{-p\frac{\sigma+\gamma}{\sigma-p+1}(C_1\delta+1)+\alpha+\frac{\delta C_0}{\sigma-p+1}-1}(\log
r)^be^{-\frac{\delta\theta}{\sigma-p+1}\log \frac{r}{2}\log\log \frac{r}{2}}\,dr.
\end{equation*}
We define $s$ as in \eqref{38}, so that from the
above inequality we obtain
\begin{equation*}
\delta^{-p\frac{\gamma+\sigma}{\sigma-p+1}}I_2\leq
C\int_{\frac{R}{2}}^{+\infty}r^{-s}(\log
r)^be^{-\frac{\delta\theta}{\sigma-p+1}\log \frac{r}{2}\log\log \frac{r}{2}}\,\frac{dr}{r}.
\end{equation*}
The change of variable $\xi=s\log \frac{r}{2}$ then
yields
\begin{equation*}
\delta^{-p\frac{\gamma+\sigma}{\sigma-p+1}}I_2\leq
C2^{-s}s^{-1}\int_{s\log\frac{R}{4}}^{+\infty}e^{-\xi}\left(\log2+\frac{\xi}{s}\right)^be^{-\frac{\delta\theta}{s(\sigma-p+1)}\xi\log\frac{\xi}{s} }\,d\xi.
\end{equation*}
By inequality \eqref{39} we have
\begin{equation*}
  \delta^{-p\frac{\gamma+\sigma}{\sigma-p+1}}I_2\leq
Cs^{-b-1}\left(\int_{s\log\frac{R}{4}}^{+\infty}e^{-\xi}\xi^b\,d\xi\right)e^{-\frac{\delta\theta}{\sigma-p+1}\log\frac{R}{4}\log\log\frac{R}{4} }
\end{equation*}
Note that by \eqref{39} we also have
$$s\log\frac{R}{4}\geq\frac{\delta}{\sigma-p+1}\log\frac{R}{4}=\frac{1}{\sigma-p+1}\frac{\log\frac{R}{4}}{\log R}\geq\frac{1}{2(\sigma-p+1)}>0$$
for $R>0$ large enough. Hence, using again \eqref{39}, we deduce that
\begin{equation}\label{40}
\begin{aligned}
\delta^{-p\frac{\gamma+\sigma}{\sigma-p+1}}I_2&\leq
C\delta^{-b-1}\left(\int_{\frac{1}{2(\sigma-p+1)}}^{+\infty}e^{-\xi}\xi^b\,d\xi\right)e^{-\frac{\delta\theta}{\sigma-p+1}\log\frac{R}{4}\log\log\frac{R}{4}}\\
&\leq
C\delta^{-b-1}e^{-\frac{\delta\theta}{\sigma-p+1}\left(\frac{1}{\delta}-\log4\right)\log\left(\frac{1}{\delta}-\log4\right)}\,=
C\delta^{-b-1}e^{\left(-\frac{\theta}{\sigma-p+1}+\frac{\theta\delta\log4}{\sigma-p+1}\right)\log\left(\frac{1}{\delta}-\log4\right)}.
\end{aligned}
\end{equation}
By inequalities \eqref{32}, \eqref{35} and \eqref{40}, for $R>0$
large enough and every $n\in\N$ we have
\begin{equation*}
\begin{aligned}
\int_{B_R} V\,|u|^{\sigma+\gamma} \,d\mu&\leq \int_M
V\,|u|^{\sigma+\gamma} \psi_n^{pk} \,d\mu\\&\leq
C\left[\delta^{-p\frac{\overline\gamma+\sigma-\delta}{\sigma-p+1}}n^{-\frac{\delta}{\sigma-p+1}}
\big(\log(2nR)\big)^b e^{-\frac{\theta\delta}{\sigma-p+1}\log (nR)\log\log(nR)}\right.\\
&\hspace{4,5cm}\left.+\delta^{-b-1}e^{\left(-\frac{\theta}{\sigma-p+1}+\frac{\theta\delta\log4}{\sigma-p+1}\right)\log\left(\frac{1}{\delta}-\log4\right)}\right],
\end{aligned}
\end{equation*}
thus passing to the $\liminf$ as $n$ tends to infinity we obtain
\begin{equation}\label{4}
\int_{B_R} V\,|u|^{\sigma+\gamma} \,d\mu\leq
C\delta^{-b-1}e^{\left(-\frac{\theta}{\sigma-p+1}+\frac{\theta\delta\log4}{\sigma-p+1}\right)\log\left(\frac{1}{\delta}-\log4\right)}.
\end{equation}
Since by our assumption $b+1-\frac{\theta}{\sigma-p+1}<0$, the function on the right-hand side in
the above inequality tends to $0$ as $\delta$ tends to $0^+$, hence
passing to the $\liminf$ as $R$ tends to infinity in \eqref{4} and
using Fatou's Lemma we conclude that
\begin{equation*}
\int_{M} V\,|u|^{\sigma+\overline\gamma} \,d\mu=0,
\end{equation*}
and thus $u=0$ a.e. in $M$.
\end{proof}

\end{document}